\documentclass[graybox]{svmult}
\usepackage{etex}

\usepackage{mathptmx}       
\usepackage{graphicx}        
\usepackage{amsmath}
\usepackage{amscd,amssymb}
\usepackage[all]{xy}
\usepackage{tikz}
\usepackage{braids}
\usetikzlibrary{matrix,arrows,decorations.pathmorphing,calc,shapes}
\usepackage[colorlinks,plainpages,backref,urlcolor=blue]{hyperref}
\usepackage{bbm}
\usepackage{mathrsfs}
\usepackage{xstring}
\usepackage{array}

\spnewtheorem{lem}[theorem]{Lemma}{\bfseries}{\itshape}
\spnewtheorem{rem}[theorem]{Remark}{\bfseries}{\upshape}
\spnewtheorem{exm}[theorem]{Example}{\bfseries}{\upshape}
\spnewtheorem{cor}[theorem]{Corollary}{\bfseries}{\itshape}
\spnewtheorem{quest}[theorem]{Question}{\bfseries}{\upshape}
\spnewtheorem{prb}[theorem]{Problem}{\bfseries}{\upshape}
\spnewtheorem{prop}[theorem]{Proposition}{\bfseries}{\itshape}

\DeclareMathOperator{\rank}{rank}

\DeclareMathOperator{\Tor}{Tor}
\DeclareMathOperator{\ab}{ab}

\DeclareMathOperator{\gr}{gr}

\DeclareMathOperator{\Hilb}{Hilb}

\DeclareMathOperator{\GL}{GL}
\DeclareMathOperator{\Aut}{Aut}

\DeclareMathOperator{\IA}{IA}
\DeclareMathOperator{\Conf}{Conf}
\DeclareMathOperator{\Mod}{Mod}

\newcommand{\N}{\mathbb{N}}
\newcommand{\R}{\mathbb{R}}
\newcommand{\Q}{\mathbb{Q}}
\newcommand{\C}{\mathbb{C}}
\newcommand{\Z}{\mathbb{Z}}

\newcommand{\RR}{\mathcal{R}}

\def\cprime{$'$}

\newcommand{\fh}{\mathfrak{h}}
\newcommand{\fg}{\mathfrak{g}}

\newcommand{\fm}{\mathfrak{m}}

\newcommand{\cP}{\mathcal{P}}

\newcommand{\cR}{\mathcal{R}}

\DeclareMathAlphabet{\pazocal}{OMS}{zplm}{m}{n}
\newcommand{\cI}{\pazocal{I}}

\newcommand{\vB}{vB}
\newcommand{\vP}{vP}
\newcommand{\wB}{wB}
\newcommand{\wP}{wP}

\newcommand{\s}{s}

\newcommand{\surj}{\twoheadrightarrow}
\newcommand{\inj}{\hookrightarrow}

\newcommand{\abs}[1]{\left| #1 \right|}
\def\dot{\mathchar"013A}  
\newcommand{\hdot}{{\raise1pt\hbox to0.35em{\Huge $\dot$}}} 

\newcolumntype{M}[1]{>{\centering\arraybackslash}m{#1}}
\newcolumntype{L}[1]{>{\arraybackslash}m{#1}}
\newcolumntype{N}{@{}m{0pt}@{}}

\begin{document}

\title*{The pure braid groups and their relatives}
\titlerunning{The pure braid groups and their relatives}

\author{Alexander~I.~Suciu and He Wang}
\institute{Alexander~I.~Suciu \at 
Department of Mathematics,
Northeastern University,
Boston, MA 02115, USA\\ 
\email{a.suciu@neu.edu}\\
Supported in part by the National Security 
Agency (grant H98230-13-1-0225) and the Simons Foundation 
(collaboration grant for mathematicians 354156)
\and 
He Wang \at 
Department of Mathematics,
Northeastern University,
Boston, MA 02115, USA \\ 
\email{wang.he1@husky.neu.edu}
}

\setcounter{minitocdepth}{1}
\maketitle

\abstract{In this mostly survey paper, we investigate the resonance varieties, 
the lower central series ranks, and the Chen ranks, as well as the 
residual and formality properties of several families of braid-like groups: 
the pure braid groups $P_n$, the welded pure braid groups $\wP_n$, the 
virtual pure braid groups $\vP_n$, as well as their `upper' variants, 
$\wP_n^+$ and $\vP_n^+$.  We also discuss several natural homomorphisms 
between these groups, and various ways to distinguish among the pure braid 
groups and their relatives.}


\section{Introduction}
\label{sect:intro}

\subsection{Cast of characters}
\label{subsec:characters}

Let $F_n$ be the free group on generators $x_1,\dots, x_n$, and let 
$\Aut(F_n)$ be its automorphism group. Magnus \cite{Magnus35} 
showed that the map $\Aut(F_n)\to \GL_n (\Z)$ which sends an automorphism 
to the induced map on the abelianization $(F_n)^{\ab}=\Z^n$ is surjective.  
Furthermore, the kernel of this homomorphism, 
denoted by $\IA_n$, is  generated 
by automorphisms $\alpha_{ij}$ and $\alpha_{ijk}$ 
($1\leq i\neq j\neq k\leq n$) which send $x_i$ to  
$x_jx_ix_j^{-1}$ and $x_ix_jx_kx_j^{-1}x_k^{-1}$, respectively, 
and leave invariant the remaining generators of $F_n$.
The subgroup generated by the automorphisms $\alpha_{ij}$ 
and $\alpha_{ijk}$ with $i<j<k$ is denoted by $\IA_n^+$.

An automorphism of $F_{n}$ is called a `permutation-conjugacy'  
if it sends each generator $x_i$ to a conjugate of $x_{\tau(i)}$, 
for some permutation $\tau\in S_n$.
The classical Artin braid group $B_n$ is the subgroup of $\Aut(F_n)$ 
consisting of those permutation-conjugacy automorphisms which fix the word 
$x_1\cdots x_n\in F_n$, see for instance Birman's book \cite{Birman74}.
The kernel of the canonical projection from $B_n$ to the symmetric group $S_n$
is the pure braid group $P_n$ on $n$ strings. As shown by Fadell, Fox, 
and Neuwirth \cite{Fadell-Neuwirth, Fadell-Fox}, a classifying space for $P_n$ 
is $\Conf_n(\C)$, the configuration space of $n$ ordered points on the 
complex line. 

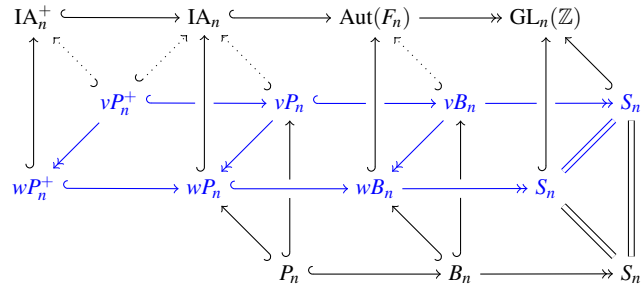
\begin{figure}[t]
\center
\begin{equation*}
\begin{tikzpicture}[baseline=(current  bounding  box.center)]
\matrix (m) [matrix of math nodes, row sep={3.8em,between origins}, 
column sep={3.8em,between origins}]{
 \textcolor{black}{\IA^+_n}&&\textcolor{black}{\IA_n}&
 &\textcolor{black}{\Aut(F_n)}&&\textcolor{black}{\GL_n(\Z)}&\\ 
 &\textcolor{blue}{\vP_n^+}&&\textcolor{blue}{\vP_n}&
 &\textcolor{blue}{\vB_n}&&\textcolor{blue}{S_n}\\
\textcolor{blue}{\wP_n^+}&&\textcolor{blue}{\wP_n}&
&\textcolor{blue}{\wB_n}&&\textcolor{blue}{S_n}&\\
 &&&\textcolor{black}{P_n}&&\textcolor{black}{B_n}&
 &\textcolor{black}{S_n}&\\};
\path[->>]
(m-1-5) edge[ ] (m-1-7)
(m-2-6) edge[blue] (m-2-8)
(m-4-6) edge[black] (m-4-8)
(m-2-2) edge[blue] (m-3-1)
(m-2-4) edge[blue] (m-3-3) 
(m-2-6)edge[blue] (m-3-5) 
;
\path[right hook->]
(m-1-1) edge[ ] (m-1-3)
(m-3-1) edge[ ] (m-1-1)
(m-1-3) edge[ ] (m-1-5)
(m-2-2) edge[dotted ] (m-1-1)
(m-2-2) edge[blue](m-2-4) 
(m-2-4) edge[blue] (m-2-6)
(m-3-1) edge[blue] (m-3-3) 
(m-4-4) edge[black] (m-4-6)
(m-4-4) edge[black] (m-3-3)
(m-4-6) edge[black] (m-3-5)
(m-3-5) edge [-,line width=7pt,draw=white](m-1-5) (m-3-5) edge[ ] (m-1-5)
(m-3-3) edge [-,line width=7pt,draw=white](m-1-3) (m-3-3) edge[ ] (m-1-3)
(m-3-7) edge  [-,line width=7pt,draw=white] (m-1-7) (m-3-7) edge[ ] (m-1-7)
(m-2-8) edge (m-1-7)
(m-4-4) edge[black] (m-2-4)
(m-4-6) edge[black] (m-2-6)
;
\path
(m-4-8) edge[black,double,double distance=1.8pt]  (m-3-7)
(m-4-8) edge[black,double,double distance=1.8pt]  (m-2-8)
(m-3-7) edge[blue,double,double distance=1.5pt]  (m-2-8);
\path[dotted,right hook->]
(m-2-6) edge (m-1-5)
(m-2-4) edge (m-1-3)
(m-2-2) edge (m-1-3)
;
\path[right hook->]
(m-3-3) edge [-,line width=7pt,draw=white](m-3-5)
(m-3-3) edge[blue] (m-3-5)
;
\path[->>]
(m-3-5) edge [-,line width=7pt,draw=white](m-3-7)
(m-3-5) edge[blue] (m-3-7)
;
\end{tikzpicture}
\end{equation*}
\caption{Braid-like groups and automorphism groups of free groups. 
\label{eq:diagram}} 
\end{figure}

The set of all permutation-conjugacy automorphisms of $F_n$ 
forms a subgroup of $\Aut(F_n)$, denoted by $B\Sigma_n$. The subgroup
$P\Sigma_n = B\Sigma_n\cap \IA_n$ is generated by the 
Magnus automorphisms $\alpha_{ij}$ ($1\le i\ne j\le n$), 
while the subgroup $P\Sigma^+_n = P\Sigma_n\cap \IA^+_n$ 
is generated by the automorphisms $\alpha_{ij}$ with $i<j$. In \cite{McCool86}, 
McCool gave presentations for the groups $P\Sigma_n$ and $P\Sigma_n^+$;   
these groups are now also called the McCool groups and the 
upper McCool groups, respectively. 

The {\em welded braid groups}\/ were introduced by 
Fenn, Rim{\'a}nyi, and Rourke in \cite{Fenn-R-R97}, 
who showed that the welded braid group $\wB_n$ is 
isomorphic to $B\Sigma_n$.  These groups, together   
with the \emph{welded pure braid groups}\/ $\wP_n\cong P\Sigma_n$ 
and the \emph{upper welded pure braid groups}\/ $\wP^+_n\cong P\Sigma^+_n$ 
have generated quite a bit of interest since then, 
see for instance \cite{ABMW, Bar-Natan14, Bardakov-B14, 
Berceanu-Papadima09, Damiani16} and references therein.
The welded pure braid group $\wP_n$ can be identified 
with group of motions of $n$ unknotted, unlinked circles in the $3$-sphere. 
As shown by Brendle and Hatcher in \cite{Brendle-Hatcher13}, 
this group can be realized as the fundamental group of 
the space of configurations of parallel rings in $\R^3$.  

A related class of groups are the \emph{virtual braid groups}\/ $\vB_n$, 
which were  introduced by Kauffman in \cite{Kauffman99} in the context 
of virtual knot theory, see also \cite{Goussarov-P-V00}. 
The kernel of the canonical epimorphism $\vB_n \to S_n$ 
is called the \emph{virtual pure braid group} $\vP_n$. 
In \cite{Bardakov04}, Bardakov found a concise presentation 
for $\vP_n$, and defined accordingly the {\em upper virtual pure braid}\/ 
group $\vP_n^+$. Whether or not the virtual (pure) braid groups 
are subgroups of $\Aut(F_n)$ is an open question that goes 
back to \cite{Bardakov04}.

The groups $\vP_n$ and $\vP_n^+$ were also independently studied 
by Bartholdi, Enriquez, Etingof, and Rains \cite{Bartholdi-E-E-R} 
and P.~Lee \cite{Lee} as groups arising from  
the Yang-Baxter equations. Classifying spaces for these groups 
(also known as the quasi-triangular groups and the triangular groups, respectively) 
can be constructed by taking quotients of permutahedra by suitable 
actions of the symmetric groups.

The groups mentioned so far fit into the diagram from Figure \ref{eq:diagram} 
(a related diagram can be found in \cite{ABMW}).  We will discuss presentations 
for these groups, various extensions and homomorphisms between them, 
as well as their centers in \S\ref{sec:braids}. 

\subsection{Lie algebras, LCS ranks, and formality}
\label{subsec:Liealgebras} 
 
To any finitely generated group $G$, there corresponds a  
graded Lie algebra,  $\gr(G)$, obtained by taking the direct sum of the 
successive quotients of the lower central series of $G$, and tensoring 
with $\C$. The {\em LCS ranks}\/ of the group $G$ are defined as the 
dimensions, $\phi_k(G)=\dim \gr_k(G)$, of the graded pieces of this 
Lie algebra.  As explained in Theorem \ref{thm:lcs koszul}, the  
computation of these ranks is greatly simplified if the group $G$ satisfies 
certain formality properties, and its cohomology algebra is Koszul.  

The set of primitive elements of the completed group algebra $\widehat{\C{G}}$ 
is a complete, filtered Lie algebra over $\C$, called the \emph{Malcev Lie algebra}\/  
of $G$, and denoted by  $\fm(G)$. By a theorem of Quillen \cite{Quillen69}, 
there exists an isomorphism of graded Lie algebras between $\gr(G)$ and $\gr(\fm(G))$.

The group $G$ is said to be \emph{graded-formal}\/ if its associated graded 
Lie algebra, $\gr(G)$, admits a quadratic presentation.
The group $G$ is said to be \emph{filtered-formal}\/ if there exists an isomorphism 
of filtered Lie algebras between  $\fm(G)$ and the degree completion of 
$\gr(G)$. Furthermore, the group $G$ is called \emph{$1$-formal}\/ 
if it is graded-formal and filtered-formal, or, equivalently, if there is a 
$1$-quasi-isomorphism between the $1$-minimal model of $G$ 
and the cohomology algebra $H^*(G,\C)$ endowed with the zero differential.
We refer to \cite{SW1} for a comprehensive study of these 
formality notions for groups. 
 
A presentation for the Malcev Lie algebra of $P_n$ was given by Kohno 
in \cite{Kohno83}, while the associated graded Lie algebra $\gr(P_n)$ 
and its graded ranks were computed by Kohno \cite{Kohno84} 
and Falk--Randell \cite{Falk-Randell}. 
It was also realized around that time that the pure braid groups $P_n$ are $1$-formal. 
As shown by Berceanu and Papadima in \cite{Berceanu-Papadima09}, the 
Malcev Lie algebras of $\wP_n$ and $\wP_n^+$ admit quadratic presentations, 
that is, the groups $\wP_n$ and $\wP_n^+$ are $1$-formal.  
Furthermore, as shown in \cite{Bartholdi-E-E-R,Lee}, 
the groups $\vP_n$ and $\vP_n^+$ are graded-formal.
On the other hand, we show in \cite{SW3} that the virtual 
pure braid groups $\vP_n$ and 
$\vP_n^+$ are $1$-formal if and only if $n\leq 3$.  

A lot is also known about the residual properties of the pure braid-like 
groups, especially as they relate to the lower central series. 
For instance, a theorem of Berceanu and Papadima \cite{Berceanu-Papadima09},  
which uses work of Andreadakis \cite{Andreadakis65} and an idea 
of Hain \cite{Hain97}, shows that the groups $\IA_n$ are residually 
torsion-free nilpotent, for all $n$.  Thus, the groups $P_n$, $\wP_n$, 
and $\wP^+_n$ also enjoy this property.  The fact that the pure braid 
groups $P_n$ are residually torsion-free nilpotent also follows from  
the work of Falk and Randell \cite{Falk-Randell, Falk-Randell88}. 
It is also known that the virtual pure braid groups $\vP_n$ and 
$\vP_n^+$ are residually torsion-free nilpotent for $n\le 3$, but it is 
not known whether this is the case for $n\ge 4$.  

\def\arraystretch{1.5}%
\begin{center}
\begin{table}[t]
\caption{Hilbert series, Koszulness, and formality of pure braid-like groups. 
\label{table:introHilb}}
\begin{tabular}{|M{1cm}|M{3.9cm}M{0.4cm}|M{2cm}M{1cm}|M{2cm}M{0.4cm}|N}
\hline 
$G$ & Hilbert series $\Hilb(H^*(G;\C),t)$ 
& & \hfill Koszulness & &\hfill $1$-Formality   & &\\[4pt]
\hline 
$P_n$
&$\prod\limits_{j=1}^{n-1}(1+jt)$ & {\hfill\vfill \scriptsize \cite{Kohno84}} 
&Yes& {\scriptsize{\vspace*{10pt}\hspace*{-6pt}\cite{Arnold69,Kohno84,SY97}}} 
&Yes& {\hfill\vfill \scriptsize \cite{Kohno83}} &\\      
\hline    
$\wP_n$& $(1+nt)^{n-1}$& { \hfill\vfill  \scriptsize \cite{Jensen-McCammond-Meier06}}  
&No (for $n\geq 4$) & { \hfill\vfill  \scriptsize \:\:\qquad\cite{Conner-Goetz15}} 
&Yes & {\hfill\vfill  \scriptsize \cite{Berceanu-Papadima09}} &\\  [10pt]
\hline 
$\wP_n^+$ & $\prod\limits_{j=1}^{n-1}(1+jt)$ 
&{\hfill\vfill  \scriptsize\cite{Cohen-P-V-Wu}}
& Yes & { \hfill\vfill  \scriptsize \:\:\qquad \cite{Cohen-Pruidze08}}
& Yes & { \hfill\vfill  \scriptsize \cite{Berceanu-Papadima09}} &\\  
\hline
$\vP_n$&$\sum\limits_{i=0}^{n-1}\binom{n-1}{i} \dfrac{n!}{(n-i)!} t^i$
&{\ \vfill \hfill\vfill\hfill \scriptsize \cite{Bartholdi-E-E-R}} 
&Yes& \ \vfill\hfill \vfill\hfill {\scriptsize \hbox{\cite{Bartholdi-E-E-R,Lee}}} 
& No (for $n\geq 4$)&{ \hfill\vfill\hfill  \scriptsize \cite{SW2}} 
& \\ [16pt]
\hline
$\vP_n^+$&$\sum\limits_{j=1}^n\dfrac{\left(\sum\limits_{i=0}^{j-1}(-1)^i\binom{j}{i}(k-i)^n\right)t^{n-j}}
 {j!}$
&{ \hfill\vfill \hfill\vfill\hfill \scriptsize \cite{Bartholdi-E-E-R}} 
&Yes& { \ \vfill\hfill\vfill \hfill \scriptsize \cite{Bartholdi-E-E-R,Lee}} 
& No (for $n\geq 4$)&{\, \hfill\vfill\hfill   \scriptsize \cite{SW2}} 
 &\\
\hline
\end{tabular}
\end{table}
 \end{center}
\def\arraystretch{1}%

\subsection{Resonance varieties and Chen ranks} 
\label{subsec:Resonance} 

We conclude our survey with a discussion of the cohomology algebras of the 
pure braid-like groups, and of two other related objects: the resonance 
varieties attached to these graded algebras, and the Chen ranks 
associated to the groups themselves.  

The cohomology algebra of the classical pure braid 
group $P_n$ was computed by Arnol{\cprime}d in his seminal paper 
on the subject, \cite{Arnold69}.  An explicit presentation for 
the cohomology algebra of the McCool group $\wP_n$ 
was given by Jensen, McCammond 
and Meier \cite{Jensen-McCammond-Meier06}, thereby 
confirming a conjecture of A.~Brownstein and R.~Lee.  Using 
different methods, 
F.~Cohen, Pakianathan, Vershinin, and Wu \cite{Cohen-P-V-Wu} 
determined the cohomology algebra of the upper McCool group $\wP^+_n$. 
Finally, the cohomology algebras of the virtual pure braid groups 
$\vP_n$ and $\vP_n^+$ were computed by
Bartholdi et al. \cite{Bartholdi-E-E-R} and Lee \cite{Lee}.

For all these groups $G$, the cohomology algebra $A=H^*(G,\C)$ is 
quadratic, i.e., it is generated in degree $1$ and the ideal of relations 
is generated in degree $2$. In fact, for all but the groups $\wP_n$, $n\ge 4$, 
the ideal of relations admits a quadratic Gr\"{o}bner basis, and so the 
algebra $A$ is Koszul.  For more details and references regarding this topic, 
we direct the reader to Table \ref{table:introHilb} and to \S\ref{subsec:cohomology}. 

Given a group $G$ satisfying appropriate finiteness conditions, 
the resonance varieties $\RR^i_s(G)$ are certain closed, homogeneous 
subvarieties of the affine space $A^1=H^1(G;\C)$, defined by means of 
the vanishing cup products in the cohomology 
algebra $A=H^*(G,\C)$.  We restrict our attention here 
to the first resonance variety $\RR_1(G)=\RR^1_1(G)$ attached to a 
finitely generated group $G$.  This variety consists of all elements 
$a\in A^1$ for which there exists an element $b\in A^1$ 
such that $a\cup b=0$, yet $b$ is not proportional to $a$.

The aforementioned computations of the cohomology algebras 
of the various pure braid-like groups allows one to determine 
the corresponding resonance varieties, at least in principle.  
In the case of the first resonance varieties of the groups 
$P_n$, $\wP_n$, and $\wP_n^+$, complete answers 
can be found in \cite{Cohen-Suciu99}, \cite{CohenD09}, 
and \cite{SW3}, respectively, while for the virtual pure braid 
groups, partial answers are given in \cite{SW2}.  
We list some of the features of these varieties in Table \ref{table:resonance}. 

\def\arraystretch{1.5}%
\begin{center}
\begin{table}[t]
\caption{Resonance and Chen ranks of braid-like groups. 
\label{table:resonance}}
\begin{tabular}{|M{1cm}|M{3cm}M{0.8cm}|M{3cm}M{0.5cm}|M{1cm}M{0.5cm}|N}
\hline 
$G$ 
&  First resonance variety $\RR_1(G)\subseteq H^1(G;\C)$&
& \hfill Chen ranks $\theta_k(G)$, $k\geq 3$ &
& Resonance--\hbox{Chen ranks} formula & &\\[4pt]
\hline 
$P_n$
&$\binom{n}{3}+\binom{n}{4}$ planes    &{ \hfill\vfill  \scriptsize  \cite{Cohen-Suciu99}}
&
$(k-1)\binom{n+1}{4}$& { \hfill\vfill  \scriptsize \cite{Cohen-Suciu95} }
& Yes &{ \hfill\vfill  \scriptsize  \cite{Cohen-Suciu99}} 
&\\   [5pt]
\hline    
$\wP_n$ &$\binom{n}{2}$  planes and $\binom{n}{4}$ linear spaces of dimension $3$  
&{ \hfill\vfill  \scriptsize\cite{CohenD09} }
&$(k-1)\binom{n}{2}+(k^2-1)\binom{n}{3}$  for $k\gg 3$
&{ \hfill\vfill  \scriptsize\cite{Cohen-Schenck15}}
& Yes  
&{ \hfill\vfill  \scriptsize\cite{Cohen-Schenck15}}
&\\  [13pt]
\hline 
$\wP_n^+$ & $(n-i)$ linear spaces of dimension $i\geq 2$
&{ \hfill\vfill  \scriptsize\cite{SW3} }
&$\sum\limits_{i=3}^k\binom{n+i-2}{i+1}+ \binom{n+1}{4}$
&{ \hfill\vfill  \scriptsize\cite{SW3}}
&No   &{ \hfill\vfill  \scriptsize\cite{SW3}}
&\\    [13pt]
\hline
$\vP_3$&$H^1(\vP_3,\C)=\C^6$
&{ \hfill\vfill  \scriptsize\cite{BMVW, SW2}}
&$\binom{k + 3}{5} + \binom{k + 2}{4} + \binom{k+1}{3} + 
 6\binom{k}{2}+ k - 2$
 &{ \hfill\vfill  \scriptsize\cite{SW2}}
 &No  
 &{ \hfill\vfill  \scriptsize\cite{SW2}}
 & \\ [18pt]
\hline
$\vP_4^+$
& $3$-dimensional non-linear subvariety of degree 6
 &{ \hfill\vfill  \scriptsize\cite{SW2}}
& $(k^3-1)+\binom{k}{2}$
&{ \hfill\vfill  \scriptsize\cite{SW2}}
&No &{ \hfill\vfill  \scriptsize\cite{SW2}}
&\\  [13pt]
\hline
\end{tabular}
\end{table}
\end{center}
\def\arraystretch{1}%

By comparing the resonance varieties of the groups $P_n$ and $\wP_n^+$, 
it can be shown that these groups are not isomorphic for 
$n\geq 4$ (cf.~\cite{SW3}); this answers a question 
of F.~Cohen et al.~\cite{Cohen-P-V-Wu}, see Remark \ref{rem:bm}.  
By computing the resonance variety 
$\RR_1(\vP_4^+)$, and using the Tangent Cone Theorem from \cite{DPS09},
we prove that the group $vP_4^+$ is not $1$-formal. In view of the retraction 
property for $1$-formality established in \cite{SW1}, we conclude that 
the groups $\vP_n$ and $\vP_n^+$ are not $1$-formal for $n\geq 4$.

The {\em Chen ranks}\/  of a finitely generated group 
$G$ are the dimensions, $\theta_k(G)=\dim \gr_k (G/G'')$,  
of the graded pieces of the graded Lie algebra associated 
to the maximal metabelian quotient of $G$.  In \cite{Chen51}, 
K.-T. Chen computed  the Chen ranks of the free groups $F_n$, 
while in \cite{Massey}, W.S. Massey gave  an alternative method for computing 
the Chen ranks of a group $G$ in terms of the Alexander invariant  $G'/G''$. 

The Chen ranks of the pure braid groups $P_n$ were computed in \cite{Cohen-Suciu95}, 
while an explicit relation between the Chen ranks and the resonance 
varieties of an arrangement group was conjectured in \cite{Suciu01}. 
Building on work from \cite{CS99, SS02, SS06} and especially \cite{PS04},  
Cohen and Schenck confirmed this conjecture in \cite{Cohen-Schenck15} 
for a class of $1$-formal groups which includes arrangement groups.  
In the process, they also computed the Chen ranks $\theta_k(\wP_n)$ 
for $k$ sufficiently large. 

Using the Gr\"{o}bner basis algorithm from \cite{Cohen-Suciu95, CS99}, 
we compute in \cite{SW3} all the Chen ranks of the upper McCool groups $\wP^+_n$.   
This computation, recorded here in Theorem \ref{thm:ChenRanksIntro}, 
shows that, for each $n\ge 4$, the group $\wP_n^+$ is 
not isomorphic to either the pure braid group $P_n$, or to the product 
$\prod_{i=1}^{n-1}F_{i}$, although these three groups share the same 
LCS ranks and the same Betti numbers.  We also provide the Chen ranks 
of the groups $\vP_n$ and $\vP_n^+$ for low values of $n$.   
The full computation of the Chen ranks of the virtual pure braid groups 
remains to be done.
 
\section{Braid groups and their relatives}
\label{sec:braids}

\subsection{Braid groups and pure braid groups}
\label{subsec:presentation}

Let $\Aut(F_n)$ be the group of (right) automorphisms of the free group $F_n$ 
on generators $x_1,\dots, x_n$. Recall that the Artin braid group $B_n$ 
consists of those permutation-conjugacy automorphisms 
which fix the word $x_1\cdots x_n\in F_n$. In particular, 
$B_1=\{1\}$ and $B_2=\Z$.  The natural inclusion 
$\alpha_n \colon B_n\inj \Aut(F_n)$ is also known 
as the Artin representation of the braid group. 

For each $1\le i <n$, let $\sigma_i$ be the braid automorphism 
which sends $x_i$ to $x_ix_{i+1}x_i^{-1}$ and $x_{i+1}$ to $x_i$, while  
leaving the other generators of $F_n$ fixed.  As shown for instance 
in \cite{Birman74},  the braid group $B_n$ is generated by the elementary braids 
$\sigma_1,\dots ,\sigma_{n-1}$, subject to the well-known relations 
\begin{equation}
\label{eq:braids}
\tag{R1}
\left\{
\begin{aligned}
&\sigma_i  \sigma_{i+1}  \sigma_i = \sigma_{i+1}  \sigma_i  \sigma_{i+1},
&1\le i \le n-2, 
\\
&\sigma_i  \sigma_j = \sigma_j  \sigma_i,&\abs{i-j}\geq 2. 
\\
\end{aligned}\right. 
\end{equation}

On the other hand, the symmetric group $S_n$ has a presentation with 
generators $s_i$ for $1\leq i\leq n-1$ and relations
\begin{equation}
\label{eq:symmetric}
\tag{R2}
\left\{
\begin{aligned}
&\s_i  \s_{i+1}  \s_i = \s_{i+1}  \s_i  \s_{i+1},
&1\le i \le n-2, 
\\
&\s_i  \s_j = \s_j  \s_i,&\abs{i-j}\geq 2. 
\\
&\s_i^2 =1,&
1\le i \le n-1;
\\
\end{aligned}\right.
\end{equation}

The canonical projection from the braid group to the symmetric 
group, which sends the elementary braid $\sigma_i$ to the 
transposition $s_{i}$, has kernel the {\em pure braid group}\/ 
on $n$ strings, 
\begin{equation}
\label{eq:pnbn}
P_n = \ker (\phi\colon B_n \surj S_n ) = B_n\cap \IA_n,
\end{equation}
where $\phi(\sigma_i)=s_i$ for $1\leq i\leq n-1$.
The group $P_n$ is generated by the $n$-stranded braids 
\begin{equation}
\label{eq:aij}
A_{ij}=\sigma_{j-1}\sigma_{j-2}\cdots\sigma_{i+1}\sigma_i^2
\sigma_{i+1}^{-1}\cdots \sigma_{j-2}^{-1}\sigma_{j-1}^{-1},
\end{equation}
for $1\leq i<j\leq n$. It is readily seen that 
$P_1=\{1\}$, $P_2= \Z$, and $P_3\cong F_2\times \Z$.  
More generally,  as shown by Fadell and Neuwirth \cite{Fadell-Neuwirth} 
(see also \cite{Falk-Randell, Falk-Randell88,Cohen-Suciu98}),
 the pure braid group $P_n$ can be decomposed
as an iterated semi-direct product of free groups,
\begin{equation}
\label{eq:pn semi}
P_n=F_{n-1} \rtimes_{\alpha_{n-1}} P_{n-1}=
F_{n-1} \rtimes  F_{n-2} \rtimes \cdots \rtimes F_1,
\end{equation}
where $\alpha_{n-1}\colon P_{n-1}\inj \Aut(F_{n})$ is the restriction of the 
Artin representation of the braid group $B_{n-1}$ to the pure 
braid subgroup $P_{n-1}$.

Work of Chow \cite{Chow} and Birman \cite{Birman74} shows that the 
center $Z(P_n)$ of the pure braid group on $n\ge 2$ strands is infinite cyclic,  
generated by the full twist braid $\prod_{1\leq i<j\leq n}A_{ij}$.  It follows that 
$P_n\cong \overline{P}_n\times \Z$, where  $\overline{P}_n = P_n/Z(P_n)$.

\subsection{Welded braid groups}
\label{subsec:welded}

The set of all permutation-conjugacy automorphisms of the free group 
of rank $n$ forms the braid-permutation group $\wB_n$.  This group 
has a presentation with generators $\sigma_i$ and $\s_i$ ($1\le i< n$) 
and relations \eqref{eq:braids} and \eqref{eq:symmetric}, as well as 
\begin{equation}
\label{eq:welded1}
\tag{R3}
\left\{
\begin{aligned}
&\s_{i}  \sigma_{j} = \sigma_{j}  \s_{i}, &|i - j| \geq 2, \\
&\sigma_{i}  \s_{i+1} \s_i=\s_{i+1} \s_{i}  \sigma_{i+1},
&1\le i \le n-2, 
\end{aligned}\right.
\end{equation}
and 
\begin{equation}
\label{eq:welded2}
\tag{R4}
\s_{i}  \sigma_{i+1}  \sigma_{i} = \sigma_{i+1}  \sigma_{i} \s_{i+1},  \quad 1\le i \le n-2.
\end{equation}

\begin{figure}
\centering
\begin{tikzpicture}[scale=0.8]
\braid [
 style all floors={fill=white},
 line width=1pt,
](braid) at (1,0)  s_1^-1 ;
\braid [
 style all floors={fill=white},
 line width=1pt,
](braid) at (5,0)  s_1 ;
\braid [
 style all floors={fill=white},
 line width=1pt,
](braid) at (5,0)  s_1^-1 ;
\fill[black] (braid) circle (4pt);
\braid [
 style all floors={fill=white},
 line width=1pt,
](braid) at (9,0)  s_1 ;
\braid [
 style all floors={fill=white},
 line width=1pt,
](braid) at (9,0)  s_1^-1 ;
\draw (9.5,-0.75) circle (0.14cm);
\draw (1.5, 0.3) node {classical};
\draw (5.5, 0.3) node {welded};
\draw (9.5, 0.3) node {virtual};
\end{tikzpicture}
\caption{Braid crossings. 
\label{fig:braids}}
\end{figure}
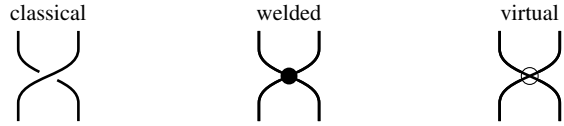 

The three types of braid crossings mentioned above are depicted in Figure \ref{fig:braids}.

The welded pure braid group $\wP_n$, also known as the group of 
basis-conjugating automorphisms  in \cite{Bar-Natan14, Fenn-R-R97, CohenD09}, 
or the McCool group in \cite{Berceanu-Papadima09},  is defined as 
\begin{equation}
\label{eq:wpndef}
\wP_n = \ker (\rho\colon \wB_n \surj S_n ) = \wB_n\cap \IA_n,
\end{equation}  
where $\rho(\sigma_i)=\rho(\s_i)=s_i$ for $1\leq i\leq n-1$.
As shown by McCool in \cite{McCool86}, 
this group is generated by the Magnus automorphisms 
$\alpha_{ij}$,  for all $1\le i\ne j\le n$, subject to the relations
\begin{align*}
&\alpha_{ij}\alpha_{ik}\alpha_{jk}=\alpha_{jk}\alpha_{ik}\alpha_{ij}, 
&\textrm{ for }  i,j,k  \textrm{ distinct},   \\
& [\alpha_{ij},\alpha_{st}]=1, &\textrm{ if }  \{i,j\}\cap \{s,t\}=\emptyset,    \\
&[\alpha_{ik},\alpha_{jk}]=1, &\textrm{ for }  i,j,k  \textrm{ distinct}.  
\end{align*}
In particular, $\wP_1=\{1\}$ and $\wP_2= F_2$.

Consider now the upper welded pure braid group (or, the upper McCool group) 
$\wP^+_n = \wP_n\cap \IA^+_n$.  This is the subgroup of $\wP_n$ generated 
by all the automorphisms $\alpha_{ij}$ with $i< j$. 
It readily seen that $\wP_1^+=\{1\}$, $\wP_2^+= \Z$,  and 
$\wP_3^+\cong F_2\times \Z$.
Furthermore, as shown by F.~Cohen et al.~\cite{Cohen-P-V-Wu}, 
the  group $\wP_n^+$ can be decomposed
as an iterated semi-direct product of free groups,
\begin{equation}
\label{eq:wpnplus semi}
\wP_n^+=F_{n-1} \rtimes_{\alpha^+_{n-1}} \wP^+_{n-1}=
F_{n-1} \rtimes  F_{n-2} \rtimes \cdots \rtimes F_1,
\end{equation}
where 
$\alpha^+_{n-1}\colon \wP_{n-1}^+\inj \Aut(F_{n-1})$ is the restriction of the 
Artin representation of $B_{n-1}$ to $ \wP_{n-1}^+$.

It follows from the previous discussion that $P_n\cong \wP_n^+$ for $n\le 3$.  
In view of this fact, a natural question (asked by F.~Cohen et al.~in \cite{Cohen-P-V-Wu}) 
is whether the groups $P_n$ and $\wP_n^+$ are isomorphic 
for $n\ge 4$.  A negative answer will be given in Corollary \ref{cor:NotIso}.
In the same circle of ideas, let us also mention the following result from \cite{SW3}. 

\begin{prop}[\cite{SW3}]
\label{prop:nonsplit}
For each $n\ge 4$, the inclusion map $\wP_n^+\inj \wP_n$ 
is not a split monomorphism. 
\end{prop}

The proof of this proposition is based upon the contrasting nature of the 
resonance varieties of the two groups. We will come back to this 
point in \S\ref{sec:coho-resonance}. 

Cohen and Pruidze showed in \cite{Cohen-Pruidze08} that
the center of the group $\wP_n^+$  ($n\ge 2$)  is infinite cyclic, 
generated by the automorphism $\prod_{1\leq j\leq n-1}\alpha_{j,n}$.  
On the other hand, Dies and Nicas showed in \cite{Dies-Nicas14} that 
the center of the group $\wP_n$ is trivial for $n\geq 2$.

\subsection{Virtual braid groups}
\label{subsec:virtual}

Closely related are the virtual braid groups $\vB_n$, the virtual pure braid groups 
$\vP_n$, and their upper triangular subgroups, $\vP_n^+$, obtained 
by omitting certain commutation relations from the respective McCool groups. 
The group $\vB_n$ admits a presentation with generators
$\sigma_i$ and $\s_i$ for $i = 1,  \dots, n-1$, 
subject to the relations \eqref{eq:braids}, \eqref{eq:symmetric}, 
and \eqref{eq:welded1}.
The virtual pure braid group $\vP_n$ is defined as the kernel of the canonical 
epimorphism $\psi\colon \vB_n\to S_n$ given by $\psi(\sigma_i)=\psi(s_i)=s_i$ 
for $1\leq i\leq n-1$, see \cite{Bardakov04}. 

A finite presentation for $\vP_n$ was given by Bardakov \cite{Bardakov04}.
The virtual pure braid group $\vP_n$ and its `upper' subgroup, $\vP_n^+$, 
were both studied in depth (under different names) by Bartholdi et al.~and Lee 
in \cite{Bartholdi-E-E-R, Lee}.  These groups are  generated by elements 
$x_{ij}$ for $i\neq j$ (respectively, for $i<j$), subject to the relations
\begin{align*}
&x_{ij}x_{ik}x_{jk}=x_{jk}x_{ik}x_{ij}, 
&\textrm{ for }  i,j,k  \textrm{ distinct}, \\
& [x_{ij},x_{st}]=1, &\textrm{ if }  \{i,j\}\cap \{s,t\}=\emptyset.
\end{align*}

Unlike the inclusion map $\wP_n^+\inj \wP_n$  from Proposition \ref{prop:nonsplit}, 
the inclusion $\vP_n^+\inj \vP_n$ {\em does}\/ admit a splitting, 
see \cite{Bartholdi-E-E-R, SW2}. 

\begin{prop}
\label{prop:groups}
There exist monomorphisms and epimorphisms making the following 
diagram commute.
\begin{equation*}
\label{eq:Lie-diagram}
\begin{tikzpicture}[baseline=(current  bounding  box.center)]
\matrix (m) [matrix of math nodes, 
row sep={3em,between origins}, 
column sep={4.7em,between origins}]{
B_n  &\vB_n    &\wB_n\\
P_n &\vP_n  &\wP_n \\};
\path[->>]
 (m-1-2) edge (m-1-3)
 (m-2-2) edge (m-2-3) 
;
\path[right hook->]
 (m-1-1) edge (m-1-2)
 (m-2-1) edge (m-2-2)
  (m-2-1) edge (m-1-1)
  (m-2-2) edge (m-1-2)
  (m-2-3) edge (m-1-3)
;
\node at (-0.8,0.7) {\small $\varphi_n$};
\node at (.8,0.7) {\small $\pi_n$};
\end{tikzpicture} 
\end{equation*}
Furthermore, the compositions of the horizontal homomorphisms are 
also injective.
\end{prop}

\begin{proof}
\smartqed
There are natural inclusions $\varphi_n\colon B_n\inj \vB_n$ and 
$\psi_n\colon B_n\inj \wB_n$ that send $\sigma_i$ to $\sigma_i$, 
as well as a canonical projection $\pi_n\colon \vB_n \surj \wB_n$, 
that matches the generators $\sigma_i$ and $s_i$ of the respective 
groups.  By construction, we have that $\pi_n\circ \varphi_n=\psi_n$.  

We claim that these homomorphisms restrict to homomorphisms between 
the respective pure-braid like groups.  Indeed, as shown Bartholdi et al.~in 
\cite{Bartholdi-E-E-R}, the homomorphism $\varphi_n$ restricts to a map 
$P_n\inj  \vP_n$, given by  
\begin{equation}
\label{eq:eting}
A_{ij}\mapsto x_{j-1,j}\dots x_{i+1,j}x_{i,j}x_{j,i}
(x_{j-1,j}\dots x_{i+1,j})^{-1}.
\end{equation}
Clearly, the projection $\pi_n$ restricts to a map $\vP_n \surj \wP_n$ 
that sends $x_{ij}$ to $\alpha_{ij}$.   Using these observations, 
together with work of Bardakov  \cite{Bardakov04}, we see that the 
homomorphism $\psi_n$ restricts to an injective map $P_n\inj \wP_n$.
\qed
\end{proof}

From the defining presentations, it is readily seen that 
$vP_2^+\cong \Z$ and $vP_2\cong F_2$, while 
$\vP_3^+ \cong \Z\ast \Z^2$.  Moreover, using a 
computation of Bardakov et al.~\cite{BMVW}, we show in \cite{SW2} 
that $\vP_3 \cong \overline{P}_4\ast \Z$.  
Consequently  $vP_2$, $vP_3$ and $vP_3^+$ have trivial centers. 

More generally,  Dies and Nicas showed in \cite{Dies-Nicas14} that the center
of the group $\vP_n$ is trivial for $n\geq 2$; 
furthermore, the center of $\vP_n^+$ is trivial for $n\geq 3$, 
with one possible exception (and no exception if Wilf's conjecture is true). 

\begin{rem}
\label{rem:bardakov}
In Lemma 6 from \cite{Bardakov04}, Bardakov states that the group $\vP_n$ splits 
as a semi-direct product of the form $V^*_n \rtimes \vP_{n-1}$, where  
$V^*_n$ is a free group.  This lemma would imply, via an easy induction 
argument, that $Z(\vP_n)=\{1\}$ for $n\ge 2$ and $Z(\vP^+_n)=\{1\}$ 
for $n\ge 3$. Unfortunately, there seems to be a problem 
with this lemma, according to the penultimate remark 
from \cite[\S{6}]{Godelle-Paris12}. 
\end{rem}

\subsection{Topological interpretations}
\label{subsec:top}

All the braid-like groups mentioned previously admit nice topological interpretations. 
For instance, 
the braid group $B_n$ can be realized as the mapping class group 
of the $2$-disk with $n$ marked points, $\Mod_{0,n}^1$, while the pure 
braid group $P_n$ can be viewed as the fundamental group of 
the configuration space of $n$ ordered points on the complex line, 
$\Conf_n(\C)=\{(z_1,\dots, z_n)\in \C^n \mid z_i\neq z_j \textrm{ for } i\neq j\}$, 
 see for instance \cite{Birman74}. 

\begin{figure}
\begin{tikzpicture}[scale=1]
\draw [->,blue,thick] (2.05,-1)
to [out=-30,in=180] (3,-1.3) to [out=0,in=-135] (4.1,-0.7);
\draw [fill=white,ultra thick,white] (3.6,-1.1) circle (0.17cm);
\draw (1.5,-1) [thick] circle (0.5cm);
\draw (3,-1) circle (0.6cm);
\draw [fill=white,ultra thick,white] (2.4,-1.22) circle (0.15cm);
\draw [-,blue,thick] (2.05,-1) to [out=-30,in=180] (3,-1.3);

\draw (6.5,-1) [thick] circle (0.5cm);
\draw (8,-1) circle (0.6cm);
\draw [fill=white,ultra thick,white] (7.4,-1.22) circle (0.12cm);
\draw [fill=white,ultra thick,white] (8.6,-1.12) circle (0.1cm);
\draw [->,blue,thick] (7.05,-1)
to [out=-30,in=180] (8,-1.3) to [out=0,in=-135] (9.1,-0.7);
\draw (2.2, 0.3) node {Classical move};
\draw (7.2, 0.3) node {Welded move};
\end{tikzpicture}
\caption{Untwisted flying rings. \label{fig:flyring}}
\end{figure}
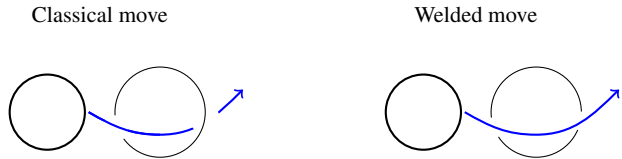

The welded braid group $\wB_n$ is the fundamental group of 
the `untwisted ring space,' which consists of all configurations 
of $n$ parallel rings (i.e., unknotted circles) in $\R^3$, see Figure \ref{fig:flyring}. 
However, as shown in \cite{Brendle-Hatcher13}, 
this space is not aspherical. 
The  welded pure braid group $\wP_n$ can be viewed as 
the pure motion group of $n$ unknotted, unlinked circles 
in $\R^3$, cf.~Goldsmith \cite{Goldsmith81}. 
The group $\wP_n^+$ is the fundamental group of the subspace consisting of 
all configurations of circles of unequal diameters in the `untwisted ring 
space,' see  Brendle--Hatcher \cite{Brendle-Hatcher13} 
and Bellingeri--Bodin \cite{Bellingeri-Bodin15}.

\begin{figure}
\begin{tikzpicture}[scale=0.9]
\draw (2,0) -- (1, 1.732) -- (2,3.464) -- (4,3.464) -- (5,1.732) -- (4,0)-- (2,0);
\draw (2,-0.3) node {\scriptsize (1,3,2)};
\draw (1.6, 1.732)node {\scriptsize(1,2,3)};
\draw (2,3.7) node {\scriptsize(2,1,3)};
\draw (4,3.7) node {\scriptsize(2,3,1)};
\draw (4.4,1.732) node {\scriptsize(3,2,1)};
\draw (4,-0.3) node {\scriptsize(3,1,2)};
\draw (4,-0.7) node { };
\draw (0.5,2) node { };
\end{tikzpicture}
\hspace{0.7cm}
\includegraphics[scale=0.16]{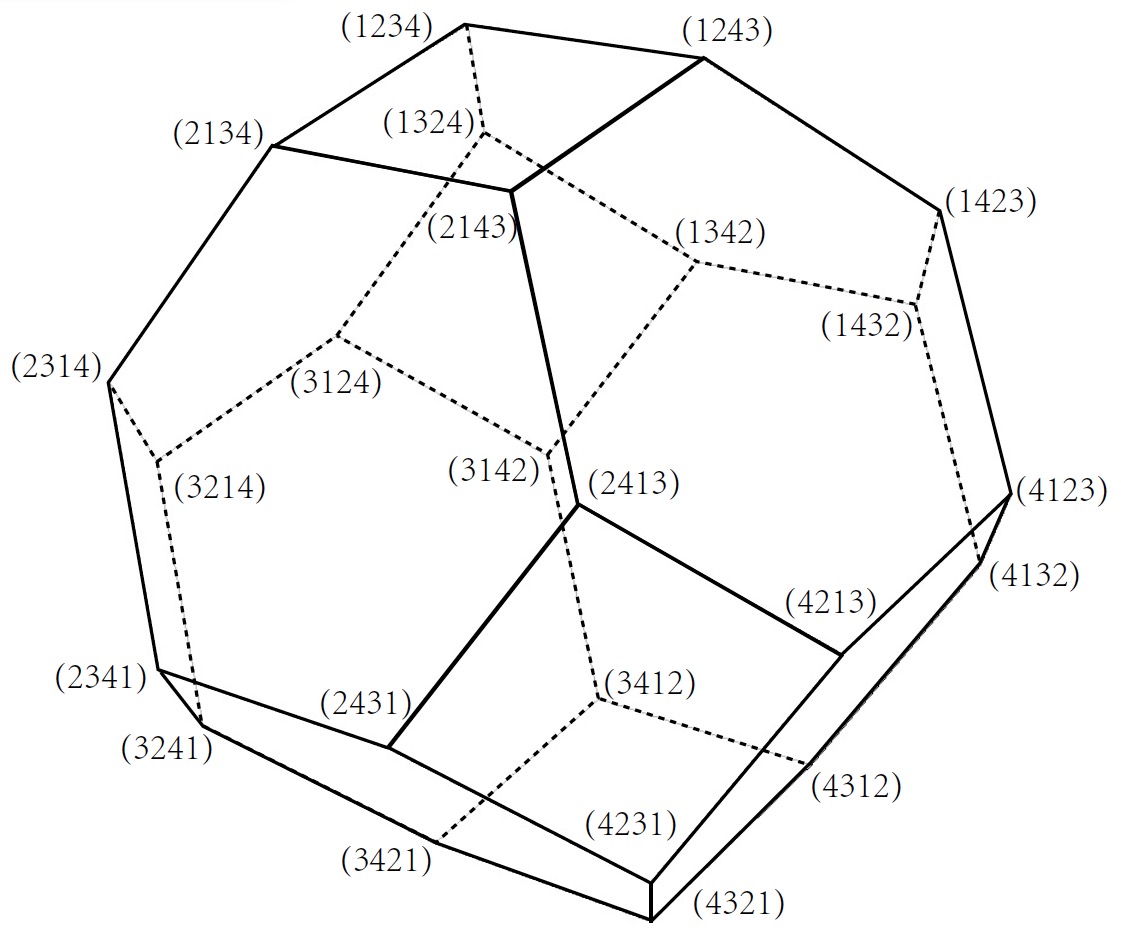}
\caption{$\mathbf{P}_3$ and $\mathbf{P}_4$. \label{fig:perm}}
\end{figure} 

A classifying space for the group $\vP_n^+$ is identified in \cite{Bartholdi-E-E-R} 
as the quotient space of the $(n-1)$-dimensional 
permutahedron $\mathbf{P}_n$ by actions of certain symmetric groups.
More precisely, let $\mathbf{P}_n$ be the convex hull of the orbit 
of a generic point in $\R^n$ under the permutation action of the symmetric 
group $S_n$ on its coordinates.  Then  $\mathbf{P}_n$ is a polytope  
whose faces are indexed by all ordered partitions of the set $[n]=\{1,\dots,n\}$; 
see Figure \ref{fig:perm}.
For each $r\in [n]$, there is a natural action of $S_r$ 
on the disjoint union of all $(n-r)$-dimensional faces of this polytope, 
$C_1\sqcup\dots\sqcup C_r$. 
Similarly, a classifying space for $\vP_n$ can be constructed as 
a quotient space of $\mathbf{P}_n\times S_n$. 
 
 \section{Cohomology algebras and resonance varieties}
\label{sec:coho-resonance}
 
 \subsection{Cohomology algebras}
\label{subsec:cohomology} 
 
Recall that the pure braid group $P_n$ is the fundamental group of the 
configuration space $\Conf_n(\C)$.
As shown by Arnol{\cprime}d in \cite{Arnold69}, 
the cohomology ring $H^*(P_n;\Z)$ is the 
skew-commutative ring generated by degree $1$ elements 
$u_{ij}$ ($1\leq i<j\leq n$), identified with the de Rham cocycles 
$\frac{dz_i-dz_j}{z_i-z_j}$. The cohomology algebras of the other pure braid-like 
groups were computed by several groups of researchers over 
the last ten years, 
see  \cite{Jensen-McCammond-Meier06,Cohen-P-V-Wu,Bartholdi-E-E-R,Lee}.
 
 We summarize these computations, as follows. To start with, 
 we denote the standard (degree $1$) generators of $H^*(\wP_n;\C)$ and 
 $H^*(\vP_n;\C)$  by $a_{ij}$ for $1\leq i\neq j\leq n$, and 
 we denote the generators of $H^*(\wP_n^+;\C)$ and $H^*(\vP_n^+;\C)$  by 
 $e_{ij}$ for $1\leq i< j\leq n$.
Next, we list several types of relations that occur in these algebras.
\begin{align*}
&u_{jk}u_{ik}=u_{ij}(u_{ik}-u_{jk}) &\textrm{ for $i<j<k$},
\label{eq:I1}\tag{I1}\\
&  a_{ij}a_{ji}=0 &\textrm{for }  i\neq j,
\label{eq:I2}\tag{I2}\\
& a_{kj}a_{ik}=a_{ij}(a_{ik}-a_{jk}) & \textrm{ for $i,j,k$ distinct},
\label{eq:I3}\tag{I3}\\
&a_{ji}a_{ik}=(a_{ij}-a_{ik})a_{jk}& \textrm{ for $i,j,k$ distinct},
\label{eq:I4}\tag{I4}\\
 & e_{ij}(e_{ik}-e_{jk})=0 & \textrm{for }  i< j< k,
\label{eq:I5}\tag{I5}\\
 &(e_{ij}-e_{ik})e_{jk}=0& \textrm{for }  i< j< k.
\label{eq:I6}\tag{I6}
\end{align*}

Finally, we record in
Table \ref{table:cohomology}
the cohomology algebras of the pure braid groups, the welded pure braid groups, 
the virtual pure braid groups, and their upper triangular subgroups.

\def\arraystretch{1.5}%
\begin{table}
\caption{Cohomology algebras of the pure braid-like groups. 
\label{table:cohomology}}
\begin{tabular}{|M{1.7cm}|M{1.7cm}|M{1.8cm}|M{1.8cm}|M{1.7cm}|M{1.8cm}|N}
\hline 
 &  $H^*(P_n;\C)$   ~~  {\scriptsize \cite{Arnold69}}& $H^*(wP_n;\C)$ ~ 
 {\scriptsize\cite{Jensen-McCammond-Meier06}} & $H^*(wP_n^+;\C)$ ~ 
 {\scriptsize\cite{Cohen-P-V-Wu}} &$H^*(vP_n;\C)$ 
 {\scriptsize\cite{Bartholdi-E-E-R,Lee}} &$H^*(vP_n^+;\C)$   
 {\scriptsize\cite{Bartholdi-E-E-R,Lee}} &\\[4pt]
\hline 
Generators & $u_{ij}$  \hbox{$1\leq i<j\leq n$} 
& $a_{ij}$ \hbox{$1\leq i\neq j\leq n$}\: 
& $e_{ij}$  \hbox{$1\leq i<j\leq n$}\:  
& $a_{ij}$ \hbox{$1\leq i\neq j\leq n$}\:  
& $e_{ij}$ \hbox{$1\leq i<j\leq n$}\: &\\[4pt]
 \hline
Relations &  \eqref{eq:I1}  
& \eqref{eq:I2} \eqref{eq:I3} 
&  \eqref{eq:I5}   
& \eqref{eq:I2} \eqref{eq:I3}  \eqref{eq:I4} 
&\eqref{eq:I5} \eqref{eq:I6} &\\[6pt]
 \hline
\end{tabular}
\end{table}
\def\arraystretch{1}%
  
The above presentations of the cohomology algebras differ slightly from 
those given in the original papers. Using these presentations, it is easily  
seen that the aforementioned homomorphisms, $f_n\colon vP_n\inj \wP_n$ and 
$g_n\colon vP_n^+\inj \wP_n^+$, induce epimorphisms in cohomology 
with $\C$-coefficients.

Let us also highlight the fact that the cohomology algebras of the pure braid-like groups 
are {\em quadratic algebras}.  More precisely, they are all algebras of the form 
$A=E/I$, where  $E$ is an exterior algebra generated in degree $1$, and 
$I$ is an ideal generated in degree $2$.  

It is also known that the  
cohomology algebras of all these groups (with the exception 
of $\wP_n$ for $n\geq 4$) are {\em Koszul algebras}.  That is 
to say, for each such algebra $A$, the ground field $\C$ admits 
a free, {\em linear}\/ resolution over $A$, or equivalently, $\Tor^A_i(\C,\C)_j=0$ 
for $i\ne j$.  In fact, it is known that in all these cases, 
the relations specified  in Table \ref{table:cohomology} 
form a quadratic Gr\"{o}bner basis for the ideal of relations $I$, a fact which  
implies Koszulness for the algebra $A=E/I$. On the other hand, 
it was recently shown by Conner and Goetz \cite{Conner-Goetz15} 
that the cohomology algebras of the groups $\wP_n$ are not Koszul 
for $n\ge 4$.  

For a summary of the above discussion, as well as detailed references 
for the various statements therein we refer to Table \ref{table:introHilb}.
 
\subsection{Resonance varieties}
\label{subsec:resonance} 

Now let $A=\bigoplus_{i\ge 0} A^i$ be a graded, graded-commutative 
algebra over $\C$.  We shall assume that $A$ is 
connected, i.e., $A^0=\C$, generated by the unit $1$. 
For each element $a\in A^1$, one can form a cochain complex, 
$(A,\delta_{a})$, known as the {\em Aomoto complex},
with differentials $\delta^i_{a}\colon A^i\to A^{i+1}$ given by 
$\delta^i_{a}(u)=a\cdot u$.  

The {\em resonance varieties}\/ of the graded algebra $A$ are the 
jump loci for the cohomology of the Aomoto complexes parametrized by 
the vector space $A^1$. More precisely, for each $i\ge 0$ and $s\ge 1$, 
the (degree $i$, depth $s$) resonance variety of $A$ is the set 
\begin{equation}
\label{eq:resvars}
\cR^i_s(A)=\{a\in A^1\mid \dim H^i(A,\delta_{a}) \geq s\}.
\end{equation}
If $A$ is locally finite (i.e., each graded piece $A^i$ is finite-dimensional),   
then all these sets are closed, homogeneous subvarieties of the affine space $A^1$. 

The resonance varieties of a finitely-generated group $G$ 
are defined as $\cR^i_s(G):=\cR^i_s(A)$, where $A=H^*(G,\C)$ 
is the cohomology algebra of the group.  If $G$ admits a classifying space  
with finite $k$-skeleton, for some $k\ge 1$, then the sets  $\cR^i_s(G)$ are 
algebraic subvarieties of the affine space $A^1$, for all $i\le k$, see 
\cite[Corollary 4.2]{PS14}.   We will focus here on the first resonance variety 
$\cR_1(G):=\cR_1^1(G)$, which can be described more succinctly as  
\begin{equation}
\label{eq:r1g}
\cR_1(G)=\{a\in A^1\mid \exists\, b\in A^1\setminus \C\cdot a 
\text{ such that }
ab=0\in A^2\}.
\end{equation}

The idea of studying a family of cochain complexes parametrized by the 
cohomology ring in degree $1$ originated from the theory of hyperplane 
arrangements, in the work of Falk  \cite{Falk97}, while the more general 
concept from \eqref{eq:r1g} first appeared in work of Matei and Suciu \cite{MS00}.  

The resonance varieties of a variety spaces and groups have been studied 
intensively from many perspectives and in varying degrees of generality, 
see for instance \cite{DPS09, Suciu12, PS14} and reference therein.   
In particular, the first resonance varieties of the groups  $P_n$, 
$\wP_n$, and $\wP^+_n$ have been completely determined in  
\cite{Cohen-Suciu99, CohenD09, SW3}.  We summarize those results as follows.

\begin{theorem}[\cite{Cohen-Suciu99}]
\label{the:res pbgroup}
For each $n\ge 3$, 
the first resonance variety of the pure braid group $P_n$
has decomposition into irreducible components given by
\begin{equation*}
\label{eq:r1pn}
 \cR_1(P_n)=    \bigcup_{1\leq i<j<k\leq n} L_{ijk}\cup 
 \bigcup_{1\leq i<j<k<l\leq n} L_{ijkl}, 
\end{equation*}
where 
\begin{itemize}
\item $L_{ijk}$ is the plane defined by the equations $\sum_{\{p,q\}\subset \{i,j,k\}}x_{pq}=0$ 
and   $x_{st}=0$ if $\{s,t\}\nsubseteq \{i,j,k\}$; 
\item $L_{ijkl}$ is the plane defined by the equations $\sum_{\{p,q\}\subset \{i,j,k,l\}}x_{pq}=0$,
$x_{ij}=x_{kl}$, $x_{jk}=x_{il}$, $x_{ik}=x_{jl}$   and 
$x_{st}=0$  if $\{s,t\}\nsubseteq \{i,j,k,l\}$. 
\end{itemize}
\end{theorem}

\begin{theorem}[\cite{CohenD09}]
\label{the:res wpn} 
For each $n\ge 2$, 
the first resonance variety of the  group $wP_n$
has decomposition into irreducible components given by
\begin{equation*}
\label{eq:r1wpn}
 \cR_1(wP_n)=\bigcup_{1\leq i<j\leq n}L_{ij} \cup \bigcup_{1\leq i<j<k\leq n} L_{ijk}, 
\end{equation*}
 where $L_{ij}$ is the plane defined by the equations $x_{pq}=0$  if $\{p,q\}\neq \{i,j\}$ and 
 $L_{ijk}$ is the $3$-dimensional linear subspace defined by the equations 
 $x_{ji}+x_{ki}=x_{ij}+x_{kj}=x_{ik}+x_{jk}=0$  and $x_{st}=0$ if $\{s,t\}\nsubseteq \{i,j,k\}$.
\end{theorem}
 
\begin{theorem}[\cite{SW3}]
\label{thm:ResonanceIntro}
For each $n\ge 2$, the first resonance variety of the group $wP_n^+$ has 
decomposition into irreducible components given by
\begin{equation*}
\cR_1(wP_n^+)=\bigcup\limits_{n-1\geq i>j\geq 1} L_{i,j},
\end{equation*}
 where $L_{i,j}$ is the linear subspace of dimension $j+1$ spanned by 
 $e_{i+1,j+1}$ and $e_{j+1,k}-e_{i+1,k}$, for $1\leq k\leq j$. 
\end{theorem}

Much less is known about the resonance varieties of the virtual pure braid 
groups.   For low values of $n$, the varieties $\RR_1(\vP_n)$ and 
$\RR_1(\vP_n^+)$ have been determined in \cite{SW2}.  For  instance, 
the decomposition $\vP_3\cong \overline{P}_4*\Z$ and known properties 
of resonance varieties of free products leads to the equality 
$\cR_1(\vP_3)=H^1(\vP_3,\C)$.  We will need the following computation 
later on. 

\begin{prop}[\cite{SW2}]
\label{prop:resonancePV4+}
The resonance variety $\cR_1(\vP_4^+)$ is the degree $6$, irreducible, 
$3$-dimen\-sional subvariety of the affine space $H^1(\vP_4^+,\C)=\C^6$ 
defined by the equations
\begin{align*}
&x_{12}x_{24}(x_{13}+x_{23})+
x_{13}x_{34}(x_{12}-x_{23})-x_{24}x_{34}(x_{12}+x_{13})=0,\\
&x_{12}x_{23}(x_{14}+x_{24})+x_{12}x_{34}(x_{23}-x_{14})+
x_{14}x_{34}(x_{23}+x_{24})=0,\\
&x_{13}x_{23}(x_{14}+x_{24})+x_{14}x_{24}(x_{13}+x_{23}) 
+x_{34}(x_{13}x_{23}-x_{14}x_{24})=0,\\
&x_{12}(x_{13}x_{14}-x_{23}x_{24})+x_{34}(x_{13}x_{23}-x_{14}x_{24})=0.
\end{align*}
\end{prop}

\section{Lie algebras and formality}
\label{sec:Liealgebras}

\subsection{The associated graded Lie algebra of a group}
\label{subsec:grlie}

The \emph{lower central series}\/ of a group $G$ is a descending sequence of normal 
subgroups, $\{ \gamma_kG\}_{k\geq 1}$, defined inductively by $ \gamma_1G=G$ and
$ \gamma_{k+1}G=[ \gamma_kG,G]$.  The successive quotients of this series, 
$\gamma_k G/\gamma_{k+1} G$, are abelian groups. 
The direct sum of these groups, 
\begin{equation}
\label{eq:grg}
\gr(G;\Z)=\bigoplus\limits_{k\geq 1} \gamma_kG/ \gamma_{k+1}G,
\end{equation}
endowed with the Lie bracket $[x,y]$ induced from the group commutator,  
has the structure of a graded Lie algebra over $\Z$.  
The {\em associated graded Lie algebra}\/ of $G$ over $\C$ is defined as 
$\gr(G)=\gr(G;\Z)\otimes_{\Z}\C$.  Before proceeding, let us recall the following 
lemma due to Falk and Randell \cite{Falk-Randell88}.

\begin{lem}[\cite{Falk-Randell88}]
\label{lem:FR}
Let $\xymatrixcolsep{11pt}
\xymatrix{1 \ar[r]& N \ar[r]& G\ar[r]& Q \ar[r]& 1}$ 
be a split exact sequence of groups, and suppose $Q$ acts trivially 
on $N/[N,N]$.  Then, for each $k\geq 1$, there is an induced split exact sequence,
$\xymatrixcolsep{11pt}
\xymatrix{1\ar[r]& \gamma_k(N) \ar[r]& \gamma_k(G)\ar[r]& \gamma_k(Q)\ar[r]& 1}$.
\end{lem}

If a group $G$ is finitely generated, its lower central series quotients 
are finitely generated abelian groups.  Set
\begin{equation}
\label{eq:phikg}
\phi_k(G)=\rank \gamma_k G/\gamma_{k+1} G,
\end{equation}
or, equivalently, $\phi_k(G)=\dim \gr_k(G)$.  These LCS ranks can be computed 
from the Hilbert series of the universal enveloping algebra of $\gr(G)$ 
by means of the Poincar\'{e}--Birkhoff--Witt theorem, as follows:
\begin{equation}
\label{eq:pbw}
\prod_{k\geq 1}(1-t^k)^{-\phi_k(G)}=\Hilb(U(\gr(G)),t).
\end{equation}

A finitely generated group $G$ is said to be {\em graded-formal}\/ if the associated 
graded Lie algebra $\gr(G)$ is quadratic, that is, admits a presentation 
with generators in degree $1$ and relators in degree $2$.  The 
next theorem was proved under some stronger 
hypothesis by Papadima and Yuzvinsky 
in \cite{Papadima-Yuzvinsky}, and essentially 
in this form in \cite{SW1}. For completeness, we sketch the 
proof.

\begin{theorem}[\cite{Papadima-Yuzvinsky, SW1}]
\label{thm:lcs koszul}
If $G$ is a finitely generated, graded-formal group, 
and its cohomology algebra, $A=H^*(G;\C)$, is Koszul, then the LCS 
ranks of $G$ are given by
\begin{equation}
\label{eq:lcsf}
\prod_{k=1}^{\infty} (1-t^k)^{\phi_k(G)} =  \Hilb(A, -t).
\end{equation}
\end{theorem}

\begin{proof}
\smartqed
By assumption, the graded Lie algebra $\fg=\gr(G)$ is quadratic, 
while the algebra $A=H^*(G;\C)$ is Koszul, hence quadratic.  
Write $A=T(V)/I$, where $T(V)$ is the tensor algebra on a 
finite-dimensional $\C$-vector space $V$ concentrated in 
degree $1$, and $I$ is an ideal generated in degree $2$.   
Define the quadratic dual of $A$ as $A^!=T(V^*)/I^{\perp}$, 
where $V^*$ is the vector space dual to $V$ and $I^{\perp}$ 
is the ideal generated by all $\alpha \in V^*\wedge V^*$ with 
$\alpha(I_2)=0$, see \cite{SY97}. 
It follows from \cite[Lemma 4.1]{Papadima-Yuzvinsky}  
that $A^!$ is isomorphic to $U(\fg)$. 

Now, since $A$ is Koszul, its quadratic dual is also Koszul. 
Thus, the following Koszul duality formula holds:
\begin{equation}
\label{eq:kdual}
\Hilb(A,t)\cdot \Hilb(A^!,-t)=1.
\end{equation}
Putting things together and using \eqref{eq:pbw} completes 
the proof.
\qed
\end{proof}

\subsection{The LCS ranks of the pure-braid like groups}
\label{subsec:lcs pb}

We now turn to the associated graded Lie algebras of the 
pure braid-like groups.  We start by listing the types of relations occurring 
in these Lie algebras:
\begin{align*}
&[x_{ij},x_{ik}]+[x_{ij},x_{jk}]+[x_{ik},x_{jk}]=0 &\textrm{for distinct $i,j,k$},
\label{eq:L1}\tag{L1}\\
& [x_{ij},x_{kl}]=0  &\textrm{for } \{i,j\}\cap\{k,l\}=\emptyset,
\label{eq:L2}\tag{L2}\\
& [x_{ik},x_{jk}]=0 & \textrm{for distinct $i,j,k$},
\label{eq:L3}\tag{L3}\\
& [x_{im}, x_{ij}+x_{ik}+x_{jk}]=0 & \textrm{for $m=j,k$ and  $i,j,m$ distinct}.
\label{eq:L4}\tag{L4}
\end{align*}

The corresponding presentations for the associated graded 
Lie algebras of the groups $P_n$, $\wP_n$, 
$\wP^+_n$, $\vP_n$, and $\vP^+_n$ are summarized in Table \ref{table:grpure}.
It is readily seen that all these graded Lie algebras are quadratic.
Consequently, all  aforementioned pure braid-like groups 
are graded-formal.

\def\arraystretch{1.5}%
\begin{table}
\caption{Associated graded Lie algebras of the pure braid-like groups. 
\label{table:grpure}}
\begin{tabular}{|M{1.7cm}|M{1.7cm}|M{1.8cm}|M{1.8cm}|M{1.7cm}|M{1.8cm}|N}
\hline 
 &  $\gr(P_n)$ \qquad   {\scriptsize \cite{Kohno83,Falk-Randell}}
 & $\gr(wP_n)$ \qquad  {\scriptsize\cite{Jensen-McCammond-Meier06}} 
 & $\gr(wP_n^+)$  \qquad {\scriptsize\cite{Cohen-P-V-Wu}} 
 &$\gr(vP_n)$ {\scriptsize\cite{Bartholdi-E-E-R,Lee}} 
 &$\gr(vP_n^+)$  {\scriptsize\cite{Bartholdi-E-E-R,Lee}} &\\[4pt]
\hline 
Generators & $x_{ij}$,  \hbox{$1\leq i<j\leq n$} & $x_{ij}$, \hbox{$1\leq i\neq j\leq n$} 
& $x_{ij}$,  \hbox{$1\leq i<j\leq n$}  & $x_{ij}$,  \hbox{$1\leq i\neq j\leq n$}  
&$x_{ij}$, \hbox{$1\leq i<j\leq n$} &\\[4pt]
 \hline
Relations &  \eqref{eq:L2} \eqref{eq:L4}  &\eqref{eq:L1}  \eqref{eq:L2} \eqref{eq:L3} 
& \eqref{eq:L1} \eqref{eq:L2} \eqref{eq:L3}   & \eqref{eq:L1} \eqref{eq:L2} 
&\eqref{eq:L1} \eqref{eq:L2} &\\[6pt]
 \hline
\end{tabular}
\end{table}
\def\arraystretch{1}%

The various homomorphisms between the pure-braid like 
groups defined previously induce morphisms 
between the corresponding associated graded Lie algebras. These 
morphisms fit into the following commuting diagram.

\begin{equation}
\label{eq:grLie-diagram}
\begin{tikzpicture}[baseline=(current  bounding  box.center)]
\matrix (m) [matrix of math nodes, 
row sep={2em,between origins}, 
column sep={3.4em,between origins}]{
\gr(\vP_n^+)  & &\gr(\vP_n)  &    \\
  && &&  \gr(P_n) \,.   \\ 
\gr(\wP_n^+) &  &  \gr(\wP_n)   & \\ };
\path[->>]
 (m-1-1) edge (m-3-1)
 (m-1-3) edge (m-3-3) 
;
\path[right hook->]
 (m-1-1) edge (m-1-3)
 (m-2-5) edge (m-1-3)
  (m-3-1) edge (m-3-3)
;
\path[->]
  (m-2-5) edge (m-3-3);
\end{tikzpicture} 
\end{equation}

As shown by Bartholdi et al.~in \cite{Bartholdi-E-E-R}, 
the morphism $\gr(\varphi_n)\colon \gr(P_n) \to \gr(\vP_n)$ 
is injective.   Using the presentations given in Table \ref{table:grpure}, we see that 
the remaining morphisms are either injective or surjective (as indicated in the 
diagram), with the possible exception of $\gr(\psi_n)\colon \gr(P_n) \to \gr(\wP_n)$, 
whose injectivity has not been established, as far as we know.  

The LCS ranks of the pure braid groups $P_n$ were computed by 
Kohno \cite{Kohno84}, using methods from rational homotopy theory, 
and by Falk and Randell \cite{Falk-Randell}, using the decomposition 
\eqref{eq:pn semi} and Lemma \ref{lem:FR}.  
The LCS ranks of the upper welded braid groups $\wP_n^+$ were 
computed by F.~Cohen et al. \cite{Cohen-P-V-Wu}, using the decomposition 
\eqref{eq:wpnplus semi} and again Lemma \ref{lem:FR}. Finally, 
work of Bartholdi et al. \cite{Bartholdi-E-E-R} and P.~Lee \cite{Lee} 
gives the LCS ranks of $\vP_n$ and $\vP_n^+$. We summarize 
these results in the next theorem.

\begin{theorem}
\label{thm:lcs vpn}
The LCS ranks of the groups $G=P_n$, $\wP_n^+$, $\vP_n$, 
and $\vP_n^+$ are given by the identity 
$\prod_{k=1}^{\infty} (1-t^k)^{\phi_k(G)} =  \Hilb(H^*(G;\C), -t)$, 
with the relevant Hilbert series given in Table \ref{table:introHilb}.
\end{theorem}

Alternatively, this result follows from Theorem \ref{thm:lcs koszul} 
once it is shown that, in all these cases, the Lie algebra $\gr(G)$ is 
quadratic and the cohomology algebra $A=H^*(G;\C)$ is Koszul. 

On the other hand, as mentioned previously, the cohomology algebras 
$H^*(\wP_n, \C)$ are not Koszul for $n\ge 4$. The (computer-aided) 
proof of this fact by Conner and Goetz \cite{Conner-Goetz15}
implies that the LCS ranks of $\wP_n$ are {\em not}\/ given by formula \eqref{eq:lcsf}.  
For instance, the formula would say that the first eight LCS ranks of $\wP_4$ are 
$12, 18, 60, 180, 612, 2\,010, 7\,020$, and $24\,480$.  Using the computations 
from \cite{Conner-Goetz15}, we see that the first seven values of $\phi_k(\wP_4)$ 
are correct, but that $\phi_8(\wP_4)=24\,490$.

\subsection{Residual properties}
\label{subsec:residual}

Let $\mathcal{P}$ be a group-theoretic property.
A group $G$ is said to be \emph{residually} $\cP$ if for any $g\in G$, $g\neq 1$,
there exists a group $Q$ satisfying property $\cP$, and an epimorphism
$\psi\colon G\to Q$ such that $\psi(g)\neq 1$.

We are mainly interested here in the residual properties related 
to the lower central series of $G$.  For instance, we say 
that the group $G$ is \emph{residually nilpotent}\/  if every 
non-trivial element can be detected in a nilpotent quotient.  
This happens precisely when the nilpotent radical of $G$ 
is trivial, that is, 
$\bigcap_{k\geq 1}  \gamma_{k}G=\{1\}$.  

Likewise, we say that a group $G$ is \emph{residually torsion-free nilpotent}\/ 
if every non-trivial element can be detected in a torsion-free 
nilpotent quotient. This happens precisely when
$\bigcap_{k\geq 1}   \tau_{k}G=\{1\}$, where 
\begin{equation}
\label{eq:tau}
\tau_{k}G=\{g\in G\mid 
\text{$g^n \in \gamma_{k}G$, for some $n\in \N$} \}.
\end{equation}

Clearly, the second property is stronger than the first. 
Nevertheless, the following holds:  if $G$ is residually 
nilpotent and $\gr_{k} (G,\Z)$ is torsion-free, for each 
$k\ge 1$, then $G$ is residually torsion-free nilpotent. In turn, 
this last property implies that $G$ is torsion-free.  Moreover, 
residually nilpotent groups are residually finite.

For a group $G$, the properties of being residually nilpotent or  
residually torsion-free nilpotent are inherited by subgroups $H<G$, 
since $\gamma_k H < \gamma_k G$ and 
$\tau_k H < \tau_k G$.  Moreover, both properties 
are preserved under direct products and free products, see Malcev \cite{Malcev49} 
and Baumslag \cite{Baumslag99}. For more on this subject, 
see also \cite{Bardakov-B09, BB09, KMP11}. 

In \cite{Andreadakis65},  Andreadakis  introduced a new filtration 
on the  automorphism group of a group $G$, nowadays called 
the Andreadakis--Johnson filtration.  This filtration is defined by setting 
\begin{equation}
\label{eq:phik}
\varPhi_k (\Aut(G))= \ker (\Aut(G) \to \Aut(G/\gamma_{k+1}(G)),  
\end{equation}
for all $k\ge 0$.  Note that $\varPhi_0(\Aut(G))=\Aut(G)$; the group 
$\cI(G)=\varPhi_1(\Aut(G))$ is called the {\em Torelli group}\/ of $G$. 

As shown by Andreadakis, if the intersection 
$\bigcap_{k\geq 1}  \gamma_{k}G$ is trivial then  the 
intersection $\bigcap_{k\geq 1}  \varPhi_{k} (\Aut (G))$ is also trivial. 
Furthermore, a theorem of L.~Kaloujnine implies that 
$\gamma_k(\cI(G)) < \varPhi_{k} (\Aut(G))$ for all $k\ge 1$, 
see e.g.~\cite{PS12}.
We thus have the following basic result.

\begin{theorem}[\cite{Andreadakis65}]
\label{thm:andreadakis}
Let $G$ be a residually nilpotent group. Then the Torelli group $\cI(G)$ 
is also residually nilpotent.  
\end{theorem}

As noted by Hain \cite{Hain97} in the case of the Torelli group of 
a Riemann surface and proved by Berceanu and Papadima 
\cite{Berceanu-Papadima09} in full generality, a stronger 
assumption leads to a stronger conclusion.

\begin{theorem}[\cite{Hain97, Berceanu-Papadima09}]
\label{thm:hbp}
Let $G$ be a finitely generated, residually nilpotent group, and suppose 
$\gr_k(G,\Z)$ is torsion-free for all $k\ge 1$.  Then the Torelli group $\cI(G)$ 
is residually torsion-free nilpotent. 
\end{theorem}

We specialize now to the case $G=F_n$.  In \cite{Magnus35}  
Magnus showed that all free groups  are residually torsion-free 
nilpotent (this also follows from the aforementioned results 
of Malcev and Baumslag).  Furthermore, P.~Hall and Magnus 
showed that $\gr(F_n,\Z)$ is the free Lie algebra 
on $n$ generators, and thus is torsion-free (see \cite[Ch.~IV, \S{6}]{Serre}). 
Therefore, by Theorem \ref{thm:hbp}, 
the Torelli group $\IA_n=\cI(F_n)$ is residually torsion-free 
nilpotent.  Hence, all its subgroups, such as $\IA^+_n$, 
$P_n$, $\wP_n$, and $\wP_n^+$ also enjoy this property.

Let us now look in more detail at the residual properties of 
the braid groups and their relatives.  We start with the classical 
braid groups.   As shown by Krammer \cite{Krammer02}
and Bigelow \cite{Bigelow01}, the braid groups $B_n$ admit faithful linear 
representations, and thus, by a theorem of Malcev, they are 
residually finite.  On the other hand, it was shown in \cite{Gorin-Lin69} 
by Gorin and Lin that $\gamma_2B_n=\gamma_3B_n$ for $n\ge 3$
(see \cite{BGG08} for an alternative proof); 
thus, the braid groups $B_n$ are not residually nilpotent for $n\ge 3$. 
Since both the welded braid group $\wB_n$ and the 
virtual braid group $\vB_n$ contain the braid group $B_n$ 
as a subgroup, we conclude that $\wB_n$ 
and $\vB_n$ are not residually nilpotent for $n\ge 3$, either 
(see also \cite{Bardakov-B09}).

A different approach to the residual properties of the pure braid groups 
was taken by Falk and Randell in \cite{Falk-Randell88}. Using the 
semi-direct product decomposition \eqref{eq:pn semi} and Lemma \ref{lem:FR}, 
these authors gave another proof that the groups $P_n$ are residually nilpotent;   
in fact, their proof shows that $P_n$ is 
residually torsion-free nilpotent, see  \cite{Bardakov-B09}. 
A similar approach, based on decomposition \eqref{eq:wpnplus semi}
provides another proof that the upper McCool groups  $wP_n^+$ 
are residually torsion-free nilpotent. 

From the work of Berceanu and Papadima \cite{Berceanu-Papadima09} 
mentioned above, we know that the full McCool groups $\wP_n$ 
are  residually torsion-free nilpotent.  For $n=3$, an even stronger 
result was proved by  Metaftsis and Papistas  \cite{MP15}, who 
showed that $\gr_k(\wP_3,\Z)$ is torsion-free for all $k$. 
Whether an analogous statement holds 
for the McCool groups $\wP_n$ with $n\ge 4$ is an open problem.

Finally, let us consider the virtual pure braid groups.   
Clearly, the groups $\vP_2^+=\Z$,  $\vP_2=F_2$, 
and $\vP_3^+=\Z*\Z^2$ are all residually torsion-free nilpotent. 
Bardakov et al.~\cite{BMVW} show that 
the group $\vP_3$ also enjoys this property. 
Alternatively, this can be seen from our decomposition 
$\vP_3 \cong \overline{P}_4 * \Z$, 
the fact that both $\overline{P}_4$ and $\Z$ have this property, and the 
aforementioned result of Malcev. 
Whether or not the groups $vP_n$ and $\vP_n^+$ are residually torsion-free nilpotent
for $n\ge 4$ is an open problem, see \cite{Bardakov-B09}.  

\subsection{Malcev Lie algebras and formality properties}
\label{subsec:malcev}

For each finitely-generated, torsion-free nilpotent group $N$, A.I.~Malcev \cite{Malcev} 
constructed a filtered Lie algebra $\fm(N)$ over $\Q$, 
which is now called the Malcev Lie algebra of $N$. Given a finitely generated group $G$,
the inverse limit of the Malcev Lie algebras of the torsion free parts of 
the nilpotent quotients $G/\gamma_i G$ for $i\geq 2$ defines the 
{\em Malcev Lie algebra}\/  of the group $G$, which we denote by $\fm(G)$. 
This Lie algebra coincides with the dual Lie algebra of the $1$-minimal 
model of $G$ defined by D.~Sullivan.
For relevant background,  we refer to \cite{SW1} and references therein. 

We will use another equivalent definition of the Malcev Lie algebra which
was given by Quillen in \cite{Quillen69}.
The group-algebra $\C{G}$ has a natural Hopf algebra structure
with comultiplication $\Delta\colon \C{G}\otimes \C{G}\to \C{G}$ given by 
$\Delta(x)=x\otimes x$ for $x\in G$. 
Let $\widehat{\C{G}}=\varprojlim_r \C{G}/I^r$ be the completion of $\C{G}$ 
with respect to the $I$-adic filtration, where $I$ is the augmentation ideal.  
An element $x\in \widehat{\C G}$ is called \textit{primitive} if 
$\widehat{ \Delta} x=x\hat{\otimes} 1+1\hat{\otimes} x$.
The Malcev Lie algebra $\fm(G)$ is then the set of 
all primitive elements in $\widehat{\C G}$, with Lie bracket 
$[x,y]=xy-yx$, and endowed with the induced filtration.

The group $G$ is said to be {\em filtered-formal}\/ if there exists an 
isomorphism of filtered Lie algebras, $\fm(G)\cong \widehat{\gr}(G)$. 
The group $G$ is \emph{$1$-formal}\/ if there exists
an filtered Lie algebra isomorphism $\fm(G)\cong \widehat{\fh}$, 
where $\fh$ is a quadratic Lie algebra. 
As shown in \cite{SW1}, a finitely generated group $G$ is $1$-formal if and only 
if it is both graded-formal and filtered-formal.  

\begin{theorem}
\label{thm:pbformal}
For each $n\ge 1$, the following hold. 
\begin{enumerate}
\item \textup{(\cite{Arnold69,Kohno83})} The pure braid group $P_n$ is $1$-formal.
\item \textup{(\cite{Berceanu-Papadima09})} The welded pure braid groups 
$\wP_n$ and $\wP_n^+$ are $1$-formal.  
\item \textup{(\cite{Bartholdi-E-E-R, Lee})} The virtual pure braid groups 
$\vP_n$ and $\vP_n^+$ are graded-formal.
\end{enumerate}
\end{theorem}

Let us now recall the following consequence of the `Tangent Cone Theorem' 
from \cite{DPS09}.

\begin{theorem}[\cite{DPS09}]
\label{thm:tangentcone}
Let G be a finitely generated, $1$-formal group. Then all irreducible 
components of $\cR_1(G)$ are rationally defined linear subspaces 
of $H^1(G,\C)$.
\end{theorem}

In view of Proposition \ref{prop:resonancePV4+} and Theorem \ref{thm:tangentcone}, 
the group $\vP_4^+$ is not $1$-formal. In fact, we have the following theorem.

\begin{theorem}[\cite{SW3}]
\label{thm:pvn}
For the virtual pure braid group $\vP_n$ 
and its upper-triangular subgroup, $\vP^+_n$, the following hold:  
both are $1$-formal for $n\leq 3$,
and both are non-filtered-formal for $n\geq 4$.
\end{theorem}

\subsection{Chen ranks}
\label{subsec:Chen}

The \textit{Chen Lie algebra}\/ of a finitely generated group $G$ is defined 
to be the associated graded Lie algebra of its second derived quotient, 
$G/G^{\prime\prime}$. The projection $\pi\colon G\surj G/G^{\prime\prime}$ induces 
an epimorphism, $\gr(\pi) \colon \gr(G)\surj \gr(G/G^{\prime\prime})$. It is readily 
verified that $\gr_k(\pi)$ is an isomorphism for $k\leq 3$.

In \cite{Chen51}, K.-T. Chen gave a method for finding a basis for 
$\gr(G/G^{\prime\prime})$ via a path integral technique for free groups.  In  
the process, he showed that the Chen ranks of the free group of rank $n$ 
are given by $\theta_1(F_n)=n$ and 
\begin{equation}
\label{eq:chen fn}
\theta_k(F_n)=\binom{n+k-2}{k} (k-1), ~ \textrm{ for } k\ge 2.
\end{equation}

As shown by Massey in \cite{Massey}, the Chen ranks of a group 
$G$ can be computed from the Alexander invariant $G'/G''$.
In \cite{Cohen-Suciu95, CS99}, Cohen and Suciu developed this method, 
by introducing the use of Gr\"{o}bner basis techniques in this context.  
As an application, they showed in \cite{Cohen-Suciu95} that 
the Chen ranks of the pure braid groups $P_n$ are given by 
\begin{equation}
\label{eq:chen pn}
\theta_k(P_n)=(k-1)\binom{n+1}{4}, ~\textrm{ for } k\geq 3.
\end{equation}

Based on this and many other similar computations, the first author 
conjectured in \cite{Suciu01} that for $k \gg 0$, the Chen ranks of 
an arrangement group $G$ are given by
\begin{equation}
\label{eq:ChenranksConjecture}
\theta_k(G)= \sum_{m\geq 2}c_m\cdot \theta_k(F_m) , 
\end{equation}
where $c_m$ is the number of $m$-dimensional components of $\cR_1(G)$.
Much work has gone into proving this conjecture, with special cases being 
verified in  \cite{SS02, PS04, SS06}.  A key advance was made in \cite{PS04}, 
where it was shown that the Chen ranks of a finitely presented, $1$-formal 
group $G$ are determined by the truncated cohomology ring $H^{\le 2}(G,\C)$. 

Using this fact, Cohen and Schenck show in \cite{Cohen-Schenck15}  that, 
for a finitely presented, commutator-relators $1$-formal group $G$, the Chen ranks 
formula \eqref{eq:ChenranksConjecture} holds, provided the components of 
$\cR_1(G)$ are isotropic, projectively disjoint, and reduced as a scheme. 
They also verify that arrangement groups and the welded pure braid groups 
$\wP_n$ satisfy these conditions.  From the Chen ranks formula 
\eqref{eq:ChenranksConjecture} and the first resonance varieties 
of $\wP_n$ in \eqref{eq:r1wpn}, they deduce that for $k\gg 1$, the 
Chen ranks of $\wP_n$ are given by
\begin{equation}
\label{eq:chenranksmccool}
\theta_k(\wP_n)=(k-1)\binom{n}{2}+(k^2-1) \binom{n}{3}. 
\end{equation}

We conjecture that formula \eqref{eq:chenranksmccool} holds for all $n$ and all $k\geq 4$.
We have verified this conjecture for $n\leq 8$, based on direct computations of the Chen ranks 
of the groups $wP_n$ in that range.

 However, the resonance varieties of 
$\wP_n^+$ do not satisfy the isotropicity hypothesis.  Nevertheless, we compute 
in \cite{SW3} the Chen ranks of these groups, using the Gr\"{o}bner basis algorithm 
outlined in \cite{Cohen-Suciu95}. The result reads as follows.
 
\begin{theorem}[\cite{SW3}]
\label{thm:ChenRanksIntro}
The Chen ranks of $\wP_n^+$ are given by
$\theta_1=\binom{n}{2}$, $\theta_2=\binom{n}{3}$, $\theta_3=2\binom{n+1}{4}$, and
\begin{equation*}
\theta_k=\binom{n+k-2}{k+1}+\theta_{k-1}=
\sum\limits_{i=3}^k\binom{n+i-2}{i+1}+ \binom{n+1}{4}
\end{equation*}
for $k\geq 4$.
\end{theorem}
 
We have seen previously that the pure braid group $P_n$, 
the upper McCool group $\wP_n^+$, and
the group $\Pi_n=\prod_{i=1}^{n-1}F_{i}$ share the same LCS ranks 
and the same Betti numbers.  Furthermore, the centers of all these 
groups are infinite cyclic, provided $n\ge 2$. 
However, the Chen ranks can distinguish these groups.
 
\begin{cor}[\cite{SW3}]
\label{cor:NotIso}
For $n\geq 4$, the pure braid group $P_n$, the upper McCool group $\wP_n^+$, 
and the group $\Pi_n$ are all pairwise non-isomorphic.
\end{cor}

\begin{rem}
\label{rem:bm}
The fact that $P_n\not\cong \wP_n^+$ for $n\ge 4$ answers in the negative  
Problem 1 from \cite[\S{10}]{Cohen-P-V-Wu}.  An alternate solution for $n=4$ was 
given by Bardakov and Mikhailov in \cite{Bardakov-Mikhailov08}, but that 
solution relies on the claim that the single-variable Alexander polynomial 
of a finitely presented group $G$ is an invariant of the group, a claim 
which is far from being true if $b_1(G)>1$.
\end{rem}

The Chen ranks of the virtual pure braid groups and their upper triangular 
subgroups are more complicated. We summarize some of our computations 
of these ranks, as follows. 

\begin{align*}
\sum_{k\geq 2} \theta_k(\vP_3^+) t^{k-2} &= (2-t)/(1-t)^3 ,\\
\sum_{k\geq 2} \theta_k(\vP_4^+) t^{k-2} &= (8-3t+t^2)/(1-t)^4 ,\\
\sum_{k\geq 2} \theta_k(\vP_5^+) t^{k-2} &=(20+15t+5t^2)/(1-t)^4,\\  
\sum_{k\geq 2} \theta_k(\vP_6^+) t^{k-2} &=(40+35t-40t^2-20t^3)/(1-t)^5.   \\
\sum_{k\geq 2} \theta_k(\vP_3) t^{k-2} &= (9-20t+15t^2-4t^4+t^5)/(1-t)^6.
\end{align*}

It would be interesting to find closed formulas for the Chen ranks of 
the groups $\vP^+_n$ and $\vP_n$, but this seems to be a very challenging 
undertaking.
 
\newcommand{\arxiv}[1]
{\texttt{\href{http://arxiv.org/abs/#1}{arXiv:#1}}}
\newcommand{\doi}[1]
{\texttt{\href{http://dx.doi.org/#1}{doi:#1}}}
\newcommand*\MR[1]{%
\StrBefore{#1 }{ }[\temp]%
\href{http://www.ams.org/mathscinet-getitem?mr=\temp}{MR#1}}

\providecommand{\bysame}{\leavevmode\hbox to3em{\hrulefill}\thinspace}
\providecommand{\MR}{\relax\ifhmode\unskip\space\fi MR }
\providecommand{\MRhref}[2]{%
  \href{http://www.ams.org/mathscinet-getitem?mr=#1}{#2}
}
\providecommand{\href}[2]{#2}

\end{document}